\documentclass{amsart}
\usepackage[utf8]{inputenc}
\usepackage[T1]{fontenc}
\usepackage[all]{xy}
\usepackage{eucal,enumerate}

\usepackage{tikz-cd}
\usepackage{amssymb}
\usepackage{mathtools}
\usepackage{hyperref}
\usepackage{epigraph}

\usepackage[a4paper,margin=1.2in]{geometry}

\title[Algebraic Integrability]{On Algebraic Integrability\\
of Vector Fields \\
with Rational Function Coefficients\\
that Separate Variables}
\author{Przemysław Grabowski}
\date{\today}

\newcommand{\cO}{\mathcal{O}}
\newcommand{\F}{\mathcal{F}}

\newcommand{\op}{\operatorname}
\renewcommand{\phi}{\varphi}

  \newtheorem{thm}{Theorem}[section]
  \newtheorem{lemma}[thm]{Lemma}
  \newtheorem{prop}[thm]{Proposition}
  \newtheorem{cor}[thm]{Corollary}

  \theoremstyle{definition} 
  \newtheorem{defin}[thm]{Definition}
  \newtheorem{remark}[thm]{Remark}
  \newtheorem{example}[thm]{Example}
  \newtheorem{problem}[thm]{Problem}
  \newtheorem{conj}[thm]{Conjecture}

\begin{document}

\begin{abstract}
    We work over a field of characteristic zero - primarily over an algebraic closure of the field of rational numbers.
    
    We prove a necessary and sufficient condition for a vector field whose coefficients are rational functions separating variables to be algebraically integrable, that is, the subring of rational functions killed by this vector field, its ring of first integrals, is of maximal possible dimension. 
    
    We do it arithmetically by reducing the vector field modulo almost all primes.
    In particular, we verify the generalized Grothendieck--Katz $p$-curvature conjecture for foliations defined by these vector fields.

    Finally, we use the outcome of that verification to provide explicit formulas for coefficients of all algebraically integrable vector fields separating variables, and their first integrals.
\end{abstract}

\maketitle
\tableofcontents

\section{Introduction}
\epigraph{If it bleeds, we can kill it.}
{Predator (1987)}

We work on the following classical algebraic problem.

\begin{problem}[Existence of First Integrals Problem]\label{problem1}
    Let $D$ be a derivation on a field $K$ over a field $k\subset K$. When does there exist an element $F\in K$, not in $k^s$, such that $D(F)=0$?
\end{problem}

We solve this problem completely for a particular class of derivations, and we explicitly describe all first integrals in terms of their coefficients.

\begin{defin}[We study these derivations]
    Let $K$ be a field of characteristic zero with a transcendental basis $x_1,x_2,\ldots, x_n$, where $n\in \mathbb{N}$, over a subfield $k$.

    We say that a derivation $D:K\to K$ over $k$ has \textbf{rational function coefficients that separate variables} if there exist non-zero rational functions in one variable $f_1,f_2,\ldots,f_n \in k(X)$ such that

    \begin{equation}\label{eq: formula - definition}
    D= f_1(x_1)\frac{\partial}{\partial x_1}+f_2(x_2)\frac{\partial}{\partial x_2}+\ldots+f_n(x_n)\frac{\partial}{\partial x_n}.
\end{equation}
\end{defin}

This class of derivations extends the case of separated variables from ODEs with a specialized type of $f_i$.
The following is a folklore concrete example from the discourse of Grothendieck--Katz $p$-curvature conjecture.

\begin{example}[Folklore Example]\label{ex: folclor example}
    A standard example is $D=x\frac{\partial}{\partial x}+ay\frac{\partial}{\partial y}$ with $a\in \overline{\mathbb{Q}}^{\times}$. 
    
    The answer: a nontrivial first integral $F$ exists if and only if $a\in \mathbb{Q}$. Moreover, some such functions can be given as $F=\left(\frac{y^{1/a}}{x}\right)^N=\op{exp}\left(N(\frac{1}{a}\op{log}(y)-\op{log}(x))\right)$ for any integer $N\ne 0$ such that $N/a$ is an integer.
\end{example}

In this paper, we generalize this answer from the above example \ref{ex: folclor example} to all $D$ of the above form \ref{eq: formula - definition} over any field $K$ of characteristic zero. Moreover, we also use our results to provide even more explicit formulas for the coefficients $f_i$ in the case of $K=\mathbb{Q}(x_1,x_2,\ldots,x_n)$ and $k=\mathbb{Q}$.

\subsection{Our Results}
For simplicity, here, we only present $n=2$, but $n\ge 2$ is analogous.
We begin with explicit formulas for $K=\mathbb{Q}(x_1,x_2,\ldots,x_n)$ and $k=\mathbb{Q}$, so $K=\mathbb{Q}(x,y)$, Proof \ref{thm: explicit, n=2, both L, over rational number}.

\begin{thm}[Explicit Main Theorem for $n=2$ over Rational Numbers]\label{thm: explicit main thm in intro}
    Let $f,g\in \mathbb{Q}(X)$ be nonzero rational functions.
    Let $D=f(x)\frac{\partial}{\partial x}+g(y)\frac{\partial}{\partial y}$ be a derivation on the field $\mathbb{Q}(x,y)$.

    Then the following are equivalent:
    \begin{enumerate}
        \item There exists a non-constant $F\in \mathbb{Q}(x,y)$ such that $D(F)=0$.
        \item One of the following is true:
    \begin{enumerate}
        \item All complex residues of $\frac{1}{f}$ and $\frac{1}{g}$ are zero.
        \item There exists irreducible polynomials $f_1,\ldots,f_n\in \mathbb{Q}[x],g_1,\ldots,g_m\in\mathbb{Q}[y]$ and rational numbers $c_1,\ldots,c_n, d_1,\ldots,d_m \in\mathbb{Q}$ such that
        \[
        \frac{1}{f}=\sum_{i=1}^n \frac{c_i f_i'}{f_i}, \hspace{1cm} \frac{1}{g}=\sum_{i=1}^m \frac{d_i g_i'}{g_i}.
        \]
        \item There exists a quadratic extension $L/\mathbb{Q}$, irreducible polynomials $f_1,\ldots,f_n\in \mathbb{Q}[x]$, $g_1,\ldots,g_m\in\mathbb{Q}[y]$ that split over $L$ into conjugate factors $f_i=f_{i1}f_{i2}$ and $g_j=g_{j1}g_{j2}$, and $c_1,\ldots,c_n\in L\setminus \mathbb{Q}$, $d_1,\ldots,d_m\in L\setminus \mathbb{Q}$ satisfying $c_i+\overline{c_i}=0,d_j+\overline{d_j}=0$ such that
        \[
        \frac{1}{f}=\sum_{i=1}^n \frac{c_i (f_{i1}'f_{i2}-f_{i2}'f_{i1})}{f_i}, 
        \hspace{1cm} 
        \frac{1}{g}=\sum_{i=1}^m \frac{d_i (g_{i1}'g_{i2}-g_{i2}'g_{i1})}{g_i}.
        \]
    \end{enumerate}  
    \end{enumerate}
\end{thm}
By adding some complexity to the conditions in the formulas for $\frac{1}{f},\frac{1}{g}$, one could obtain similar formulas for any field $K$. The condition will be a combination of (2b) and (2c) with a bit more control over $c_i,d_i$. We omit that. It is not worth it.

This explicit result follows from a combination of a simple Galois theory argument and our main result, which is the strongest elementary evidence towards the generalized Grothendieck--Katz $p$-curvature conjecture \ref{conj: alg int conj} so far. The general statement is Theorem \ref{thm: main theorem n>=2}, but the following is stated for $k=\overline{\mathbb{Q}}$ to remain the classical phrase ``for almost all primes'' instead of talking about open subsets of integral models. 

\begin{thm}[Main theorem for $n=2$]\label{thm: main thm in intro}
    Let $f,g\in \overline{\mathbb{Q}}(X)$ be nonzero rational functions.
    Let $PFD(\frac{1}{f})=P_1+L_1+H_1$ and $PFD(\frac{1}{g})=P_2+L_2+H_2$ be partial fraction decompositions over $\overline{\mathbb{Q}}$, see Theorem \ref{thm: PFD}, where $L_i=\sum_{j=1}^{d_i}\frac{c_{i,j}}{X-\alpha_{i,j}}$ for $i=1,2$.
    Let $D=f(x)\frac{\partial}{\partial x}+g(y)\frac{\partial}{\partial y}$ be a derivation.

    Then the following are equivalent:
    \begin{enumerate}
        \item For almost all primes, the derivation $D$ modulo $p$ is $p$-closed.
        \item One of the following is true:
    \begin{enumerate}
        \item We have $L_1=L_2=0$.
        \item We have $P_1=P_2=H_1=H_2=0$, and 
        
        there exists $\eta\in \overline{\mathbb{Q}}$, $\eta\ne 0$, such that for all $i,j$ we have $\eta c_{i,j}\in {\mathbb{Q}}$.
    \end{enumerate}
        \item There exists a non-constant $F\in \overline{\mathbb{Q}}(x,y)$ such that $D(F)=0$.
    \end{enumerate}
    Moreover, we can write explicit formulas for the killed element:
    \begin{itemize}
        \item (2a) We can take: $F\coloneqq \int \frac{dx}{f}-\int \frac{dy}{g}$.
        \item (2b) We can take: $G \coloneqq \op{exp}\left(N\eta \left(\int \frac{dx}{f}-\int \frac{dy}{g}\right)\right)$, where $N\ne 0$ is any integer such that $N\eta c_{i,j}\in \mathbb{Z}$.
    \end{itemize}
\end{thm}

The equivalence between (2) and (3) looks like a result one might have guessed to be from the 18th or 19th century; however, it was only proved in 2025 using the language of differential fields \cite{rational_first_for_separated_2025}, and only for $n=2$. I believe that their methods can prove the equivalence of (2) and (3) for any $n\ge 2$ simply by considering many equations at once, not only one. For us, this possibility is a corollary of our work, not the way we do it. 
Our proof goes by establishing that (1) and (2) are equivalent, (3) implies (1), and (2) implies (3). We prove it arithmetically by studying $D$ modulo almost all primes. We do it for all $n\ge 1$. (For $n=1$, it requires a different formulation: Theorem \ref{thm: main: n=1}.)

For $n=2$, condition (3) is equivalent to $D$ defining an algebraically integrable foliation, but for $n>2$, we need to put ``is algebraically integrable'' instead of ``there is a first integral $F$'' in (3) in our theorem. We give a condition for the existence of $F$ in Corollary \ref{cor: the actual answer}, where we introduce a notion of m-saturation. Our answer suggests that the existence of $F$ for a general $D$ is a wilder problem to study than full algebraic integrability.

This paper seems most connected in methods and topics to the papers \cite{Honda_p_conjecture}, \cite{effective_p_conjecture_2026} (a revisit of the former), and the preprint  \cite{rational_first_for_separated_2025}. These three papers use residues of rational functions as a key component of their algebraic integrability studies. This seems to be a key part of any discussion of integrating rational functions, known since Leibniz and Bernoulli, i.e., the partial fraction decomposition; we recall it in \ref{thm: PFD}. No surprises here.

We hope our results will contribute to understanding how to work with $p$-powers of derivations; our approach is based on our theorems about how this works on curves (Proposition \ref{prop: computing p-powers on curves}). The general computation is straightforward but hard to follow, and it takes longer as $p$ increases. 

Hopefully, there exists a way to ``decompose'' derivations into pieces and track what $p$-powers do and how this relates to algebraic integrability, like heights relate to algebraic relations between numbers. This could lead to a gadget giving bounds for degrees of solutions in terms of the formulas defining the derivation, which would solve Conjecture \ref{conj: alg int conj}. (We can extract such bounds from our main theorem.) Another idea around this conjecture is to produce a functor of solutions for $\infty$-foliations as for $D$-modules (I believe these should be somehow between algebraic and formal, because they are ``limits'' of things up to $p^n$-powers; this requires better stacks modulo $p$), lift them, and hope that their sets of lifts intersect nontrivially. Or to take the conjecture as a question: how to show that these two ind-constructible sets are equal? And do some Baire-category-like argument with some like-Borel sets. These are some loose thoughts. The conjecture is difficult, and it requires new tools. For these, some inspiration is required. That's all.

\subsubsection{More Related Literature}

The main perspectives on Problem \ref{problem1} in current mathematics fall under three broad labels, and many branches of these. Our citations are incomplete; these have been active problems for hundreds of years, starting around the time Descartes connected algebra with geometry, and they admit multiple layers of iteration. Nevertheless, the following should be enough to connect with current understanding and identify the rest for any related purposes.
\begin{itemize}
    \item integration in finite terms \cite{rosenlicht_finite_terms_1972}, which is close to differential Galois theory \cite{diff_galois_2003},
    \item Poincaré-Painlevé problem on foliations \cite{neto_poincare_2002}, \cite{pereira_poincare_2002}, 
    \item and finite monodromy problem with a Grothendieck--Katz $p$-curvature conjecture as a proposed solution \cite{GK_katz}, \cite{Honda_p_conjecture}, \cite{GK_bost}, \cite{GK_andre},   \cite{GK-survey}, \cite{lam2026pcurvaturenonabeliancohomology}, \cite{effective_p_conjecture_2026}.
\end{itemize}
  All these roads lead to working on algebraic integrability of foliations and Conjecture \ref{conj: alg int conj}. The preprint \cite{TShBE-conjectureF} that introduced the general conjecture has been retracted from public access, but it is still circulating and motivating new research.

\subsubsection{Acknowledgment}
This work was done by the author while being hired as a Simons postdoc in mathematics (2025-2028) at Kyiv School of Economics, Ukraine. I want to thank Michael F. Singer and Henryk Żołądek for helping me determine whether the main result for $n=2$ equivalence (2) and (3) was known. Eventually, the paper \cite{rational_first_for_separated_2025} was identified. Therefore, I assume this is not a classical result by Liouville, or... Euler or something. I also want to thank Asem Abdelraouf for our ramen discussion, which led to Lemma \ref{lem: ramen lemma}.

\subsubsection{AI disclosure} No part of this paper was written by any AI. All mistakes are mine. (Though I used Grammarly for English grammar. It uses AI, but it is not invasive generative AI.) Nevertheless, free AI models, Google and ChatGPT, were used as ``googling/coding expansions'': for computing explicit examples such as explicit formulas for $\op{tan}(\sqrt{a}F)$ for various $a,b$ in Example \ref{ex: x^n+a}, to quickly recall well known results and formulas, and trying to learn what is known about Problem \ref{problem1} in the literature. It was very good at bluffing solutions in a quasi-language of integration in finite terms; thus, I initially expected it was a classical result buried somewhere. My current understanding is that it is not, and this kind of reasoning is simply expected from anyone when asked.

\section{Context: Generalized Grothendieck--Katz \texorpdfstring{$p$}{p}-Curvature Conjecture}
We explain what an integral model is, what the Kronecker--Chebotarev theorem says about roots of a polynomial from its modulo $p$ data, and what the main conjecture motivating this paper is. 

\subsection{Integral Models and Good Reductions}

We work with a derivation $D$ on a field $K$ over a subfield $k$. This setup does not allow us to go modulo any prime, because $K$ is a field. To be able to go modulo any $p$, we can take a subring $A$ of $K$ over which $D$ is defined and then do this operation for an ideal $p$ of $A$. This is possible because $D$ requires finitely many symbols to be defined. This ring $A$ is called an integral model of $K$. It can be chosen to be flat. It can be chosen to preserve properties such as ``these elements are a separable transcendental basis''. Such a model is a gadget to move $D$ around and collect information about it from more exotic places, where it interacts with exotic structures; for us, its $p$-powers, which modulo $p$ are again derivations. And we have enough info from enough places, then we know all about $D$, like in the Yoneda lemma.

We do not do anything smart or novel with these integral models, so we only give references where one can read about it: \cite[Chapter 10]{Liu_AG_book}, \cite[Chapter 0C2P]{stacks-project}. We will only outline their usage in our arguments. Most of the actual usage is in the following section only anyway.

\subsection{Kronecker Theorem Plus Plus}

Here is a useful tool. It is rather folklore. It arises as a combination of the Kronecker theorem \cite[Theorem 1.1.]{effective_p_conjecture_2026}, \cite[Section 24.5]{field-arithmetic-2023}, the Chebotarev density theorem \cite[Chapter 7]{field-arithmetic-2023}, and the Hilbert irreducibility theorem \cite[Chapter 13 and 14]{field-arithmetic-2023}. Kronecker theorem is often presented as a corollary of the Chebotarev theorem.

Here we have a criterion for an abstract element from any field of characteristic zero to be a rational number.

\begin{thm}[Kronecker theorem ++]\label{thm: kronecker ++}
    Let $P\in \mathbb{Q}(t_1,\ldots,t_r)[X]$ be an irreducible polynomial. 
    If modulo all maximal primes $q$ from an open dense subset of an integral model of $\mathbb{Q}(t_1,\ldots,t_r)$ it has a root in $\mathbb{F}_p$, where $p>0$ is the characteristic of $\kappa(q)$, then $P$ is linear, and its root is a rational number.

    Consequently, if $K$ is a field of characteristic zero, $f\in K$, and $f$ modulo $q$ for all maximal ideal $q$, in an open dense subset of a suitable integral finite-type model of a subfield $f\in K'\subset K$, satisfies $X^p=X$, where $p>0$ is the characteristic of $\kappa(q)$, then $f\in \mathbb{Q}$.
\end{thm}

\begin{proof}
    Let an integral finite-type model be chosen. This is an affine scheme.
    
    By the Hilbert irreducibility theorem, there exists an integral point $a=(a_1,\ldots,a_r)\in \mathbb{Z}^r$ of the model such that $P(a)\in \mathbb{Q}[X]$ has the same degree as $P$ and is irreducible.

    The polynomial $P(a)\in \mathbb{Q}[X]$ satisfies the assumptions of the Kronecker theorem \cite[Theorem 1.1.]{effective_p_conjecture_2026}, thus $P(a)$, and so $P$ too, is of degree $1$. If the root of $P$ in $\mathbb{Q}(t_1,\ldots,t_r)$ is not rational, then we can evaluate $t_i\in \overline{\mathbb{F}_p}$ for some $p>0$ in a way that the reduction is not in $\mathbb{F}_p$. A contradiction. Thus $P$ is linear, and its root is a rational number.

    The final part follows from the first part applied to $P$ that is the minimal polynomial of $f$ over some transcendental elements $t_1,\ldots,t_r\in K$ such that $f$ is algebraic over $\mathbb{Q}(t_1,\ldots,t_r)$.
\end{proof}

\subsection{Algebraic Integrability Conjecture}

The following conjecture is a proposed answer to a finite monodromy problem and Poincaré problem due to Taylor, Shephard-Barron, and Ekedahl. They did it in their unpublished preprint \cite[Conjecture F]{TShBE-conjectureF}. It is a generalization of the Grothendieck--Katz $p$-curvature conjecture. So, one could call it by amalgamating all these names, and some more, but this would be ugly. Therefore, we restate the conjecture, give it a neutral name, and prove some lemmas.

\begin{conj}[Algebraic Integrability Conjecture]\label{conj: alg int conj}
    Let $K/\overline{\mathbb{Q}}(x_1,\ldots,x_n)$ be an algebraic extension of fields. Let $\F$ be a foliation on $K$, which admits a model over the ring of integers $\cO_L$, where $L$ is a finite extension of $\mathbb{Q}$.

    The following are equivalent:
    \begin{enumerate}
        \item The foliation $\F$ is algebraically integrable, i.e., it is of the form $\F=\op{Der}_{W}(K)$ for a subfield $\overline{\mathbb{Q}}\subset W\subset K$.
        \item For almost all primes $p$ and all maximal ideals $q$ over $p$ in $\cO_L$ the reduction modulo $q$ of $\F$ is $p$-closed, i.e., it defines a $1$-foliation.
    \end{enumerate}
\end{conj}

For the terminology of $1$-foliations and a general theory of purely inseparable subfields from the point of view of foliations, I recommend my PhD thesis \cite{Grabowski_PhD_thesis} or its arXiv version \cite{grabowski2025powertowerspurelyinseparable}.

\begin{prop}[The easy direction]\label{prop: easy direction AlgIntConj}
    For a foliation on an $n$-dimensional variety $V$ over a field of characteristic zero, being algebraically integrable implies being $p$-closed for almost all maximal ideals $q$, where $p>0$ is the characteristic of the residue field of $q$, for any suitable integral model over which it is performed.
\end{prop}
\begin{proof}
    Let $K$ be the generic point of $V$.
    
    If $\F$ is of rank $r$ and $\F=\op{Der}_{W}(K)$, then the transcendence degree of $W$ is $n-r$. Let $F_1,F_2,\ldots,F_{n-r}$ be a transcendental basis of $W$. 

    We can choose a model such that for almost all $q$ the reduction of $\F$ is of rank $r$ and $\bigwedge dF_i \ne0$ modulo $q$, i.e., the reductions $F_i$ are still a separable transcendental basis of a subfield.
    
    Finally, a derivation $D$ belongs to $\F$ modulo $q$ (over the good ideals $q$) if and only if it kills the reductions of $F_1,\ldots,F_{n-r}$. Indeed, $\F$ modulo $q$ satisfies this, and it has rank $r$, and since the derivations killing the reductions of $F_i$ form a $1$-foliation $\mathcal{G}$ of rank $r$, and $\F\subset \mathcal{G}$, thus $\F=\mathcal{G}$. Therefore, it is $p$-closed modulo all these $q$.
\end{proof}

\begin{lemma}[Algebraic integrability only depends on the field of definition]\label{lem: eventually alg is alg: rank 1 for dim 2}
    Let $K'/K$ be a finite extension of fields of characteristic zero. 
    \begin{enumerate}
        \item Any foliation $\F$ on $K$ over $k$ uniquely extends to a foliation on $K'$ over $k$, where $k$ is a subfield.
        \item If $\F$ is of rank $1$, and $D$ is a basis of $F$, then if there exists $F\in K'$ non-constant such that $D(F)=0$, then there exists $G\in K$ non-constant such that $D(G)=0$.
        \item If $\F$ is algebraically integrable over $K'$, then it is algebraically integrable over $K$.
    \end{enumerate}
\end{lemma}

\begin{proof}
    The first item is universal: any derivation extends uniquely along a separable extension by basic properties of K{\"a}hler differentials.

    The second item. Let $P=X^n+a_{n-1}X^{n-1}+\ldots +a_0$ be the minimal polynomial of $F$ over $K$, so $a_i\in K$. We have $P(F)=0$. We apply $D$.
    \[
    0=D(0)=D(P(F))=P'(F)D(F)+(D(1)F^n+D(a_{n-1})F^{n-1}+D(a_{n-2})F^{n-2}+\ldots+D(a_0)).
    \]
    Since $D(F)=0$, $D(1)=0$, and the degree of $D(a_{n-1})F^{n-1}+D(a_{n-2})F^{n-2}+\ldots+D(a_0)$ is $<n$, we get that
    for all $i=0,\ldots,n-1$ we have $D(a_i)=0$.
    If all of $a_i$ are constant, then $F$ would be from $k'\supset k$, a finite extension, i.e., constant. A contradiction: one of the $a_i$ is not constant, so we can set it to be $G$. 

    The third item generalizes the second. It is proved the same way. Take a transcendental basis of the first integrals of $\F$; then the coefficients of their minimal polynomials are killed. These elements also have the same transcendence degree as the first integrals in $K'$, because after a separable extension they are that field. They are in $K$ and killed, so we are done.

    This finishes the proof.
\end{proof}

\section{Computing \texorpdfstring{$p$}{p}-Powers On Curves}

\subsection{General Theory For Curves}

\begin{defin}
    A \textbf{generic point of a curve} is a field $K$ of positive characteristic $p>0$ such that the extension $K/k$, where $k=\bigcap_{n\ge 1} K^{p^n}$, is a finitely generated field extension of transcendental degree $1$.
    
    Concretely, it is equivalent to $K$ being isomorphic to the following construction: take a perfect field $k$, take rational functions in one variable over $k$, i.e., $k(x)$, take a finite separable extension $K/k(x)$. Any choice of such $x$ is called a coordinate for $K$.
\end{defin}

\begin{prop}[Derivations and $p$-Powers]\label{prop: identities for p-lie} Let $K$ be a field of characteristic $p>0$. Let $x\in K$. Let $D,E$ be derivations on $K$.

    \begin{enumerate}
        \item The $p$-times composition 
        $D^{\circ p}=D^p$
        is a derivation on $K$.
        \item Derivations kill $p$-powers: 
        $$D(x^p)=0.$$
        In particular, all derivations are derivations over $k$.
        \item (Jacobson's formula) There exists a universal Lie polynomial $s$ such that 
        $$\left(D+E\right)^p=D^p+E^p+s(D,E).$$
        \item (Hochschild's formula) We have 
        $$\left(xD\right)^p=x^pD^p+\left(xD\right)^{\circ (p-1)}(x)D.$$
        \item We have
        \[
        [D,xE]=D(x)E+x[D,E].
        \]
    \end{enumerate}
\end{prop}

\begin{proof}
    The first two are elementary. Let $y\in K$.
    \begin{enumerate}
        \item $D^p(xy)=\sum_{i=0}^p \binom{p}{i}D^i(x)D^{p-i}(y)=xD^p(y)+yD^p(x)$
        \item $D(x^p)=px^{p-1}D(x)=0$.
    \end{enumerate}
    
    The third one is called Jacobson's formula; we quote \cite[(5.2.4.)]{Katz_nilpotent_connections_1970}.

    For the fourth one, we quote \cite[Lemma 1]{Hochschild_1955}. This lemma proves exactly this identity in a greater generality. (Watch out, in \cite{Katz_nilpotent_connections_1970} a similar formula is introduced, but the coefficient next to $D$ there is wrong. You can check it by hand for $p=3$.)

    The fifth one is a standard identity for Lie brackets.
\end{proof}

\begin{prop}[$p$-Powers of Derivations on Curves]\label{prop: computing p-powers on curves}
    Let $K$ be a generic point of a curve with a coordinate $x$. 
    Let $f\in K$. We define a derivation on $K$ by
    \[
    D=f\frac{\partial}{\partial x}
    \]
    Then, its $p$-power is given by 
    \[
    D^{\circ p}=-b_{p-1} D
    \]
    where $b_{p-1}\in K^p$ is the unique coefficient in the expansion in the $p$-basis $\frac{x^{j}}{(j)!}$ for $0\le j<p$
    \[
    f^{p-1}=b_0+b_1\frac{x}{1!}+\ldots +b_{p-1}\frac{x^{p-1}}{(p-1)!},
    \]
    where $b_i\in K^p$.
\end{prop}

\begin{proof}
    Let $g=f^{p-1}$. We use Proposition \ref{prop: identities for p-lie} (4) for $gD$ and simplify with (2) and (5):
    \begin{align*}
        (gD)^p&=g^pD^p+(gD)^{p-1}(g)D\\
        (f^p\frac{\partial}{\partial x})^p&=f^{p(p-1)}D^p+(f^p\frac{\partial}{\partial x})^{p-1}(g)D\\
        (\frac{\partial}{\partial x})^p&=f^{p(p-1)}D^p+f^{p(p-1)}(\frac{\partial}{\partial x})^{p-1}(g)D\\
        0&=f^{p(p-1)}D^p+f^{p(p-1)}b_{p-1}D\\
        0&=f^{p(p-1)}(D^p+b_{p-1}D)
    \end{align*}
    
If $f=0$, then $D=0$, $b_{p-1}=0$, and $D^p=0$, so the formula holds. 

If $f\ne 0$, then $D^p+b_{p-1}D=0$, so we get $D^p=-b_{p-1}D$.
\end{proof}
\begin{remark}[Why the minus sign?]
    By Wilson's theorem, we have $(p-1)!\equiv -1 \ (\op{mod} p)$, thus, to get the coefficient $b_{p-1}$ we can use any basis $a_ix^i$, $i=0,1,\ldots,p-1$, of $K$ over $K^p$, where $a_i\in K^p, a_i\ne 0, a_{p-1}=-1$. 
    
    Consequently, if we write $f^{p-1}=c_0+c_1x+\ldots+c_{p-1}x^{p-1}$, then we have $D^p=c_pD$. This is actually the most common way people write it down. Nevertheless, I find putting factorials where they should be a moral duty. (This is a joke.)
\end{remark}

\subsection{For Rational Curves} We work with $K=k(x)$, where $k$ is a perfect and $\op{char}(k)=p>0$.

\begin{thm}[Partial Fraction Decomposition, {\cite[Theorem 2]{PFD_euclidean_1989}}]\label{thm: PFD}
    Let $k$ be a field of any characteristic.
    Let $f\in k(x)$ be a rational function. Let $f=\frac{a}{b}$ be a presentation of $f$, where $a,b\in k[x]$ and $(a,b)=1$. Let $b=\prod_{i=1}^n p_i^{m_i}$ be a decomposition into pairwise distinct irreducible factors. Then, there exists the \textbf{partial fraction decomposition} (PFD) of $f$:
    \begin{itemize}
        \item A polynomial $Q\in k[x]$.
        \item For each $1\le i\le n$ and $1\le j\le m_i$ a polynomial $s_{ij}\in k[x]$ with $\op{deg}(s_{ij})<\op{deg}(p_i)$.
    \end{itemize}
    These polynomials are\textbf{ unique} such that the following equality holds:
    \[
    f=Q+\sum_{i=1}^n\sum_{j=1}^{m_i} \frac{s_{ij}}{p_i^j}.
    \]
    In particular, if $k$ is algebraically closed, then, without any loss of generality, we can assume that $p_i=x-a_i$, so $s_{ij}\in k$ and we define:
    \begin{enumerate}
        \item Polynomial Part: $P(f)=Q$
        \item Logarithmic Part: $L(f)=\sum_{i=1}^n\frac{s_{i1}}{x-a_i}$, and we call $s_{i1}$ the residue of $f$ at $a_i$.
        \item Higher Part: $H(f)=\sum_{i>1}^n\sum_{j=1}^{m_i} \frac{s_{ij}}{{(x-a_i)}^j}$
    \end{enumerate}
    In total, we have $f=PFD(f)=P(f)+L(f)+H(f)$.
\end{thm}

\begin{remark}
    The separation of PFD into P,L,H is dictated by what indefinite integrals of these pieces are. For $P$, $\int P$ is a polynomial, so a rational function; for $H$, $\int H$ is a rational function; for $L$, $\int L$ is a logarithm of ``almost'' a rational/algebraic function.
\end{remark}

\begin{remark}
    The existence follows from the Euclidean algorithm, see \cite{PFD_euclidean_1989}. The uniqueness is true because, for $k[x]$, this algorithm always yields unique $q,r$ in any division $a=bq+r$. The uniqueness of PFD is likely to be true only for rings $k[x]$, because that is true for the Euclidean algorithm.

    The partial fraction decomposition and the Hermite reduction mentioned in \cite{rational_first_for_separated_2025} are basically the same concepts. Though Hermite reduction is more about how to compute these things in practice without $k=\overline{k}$. The references to specific algorithms can be found inside that paper.
\end{remark}

We introduce a bracket notation for a coefficient with respect to a $p$-basis. A general discussion about bracket notations can be found in \cite{on_bracket_Knuth_1994}.

\begin{defin} We introduce a notation for an $n$-coefficients of an element.
\begin{enumerate}
    \item Let $k$ be a field. Let $K$ be $k[x]$, or $k((x))$ Laurent series. Let $n$ be an integer.
    Let $f\in K$ be given by $f=\ldots+a_{-1}x^{-1}+a_0+a_1x+a_2x^2+\ldots$, where $a_i\in k$. We define:
    \begin{equation}
        [x^{n}]f \coloneqq a_n \in k.
    \end{equation}
    \item Let $K$ be a generic point of a curve with a coordinate $x$, or $k((x))$.
    Then, $1,x,x^2,\ldots,x^{p-1}$ is a $p$-basis of $K$ over $K^p$.
    Let $n\in \{0,1,\ldots,p-1\}$. Let $f\in K$. It has a unique presentation of the form $f=a_0+a_1x+a_2x^2+\ldots+a_{p-1}x^{p-1}$, where $a_i\in K^p$. We define:
    \begin{equation}
        \widetilde{[x^{n}]}f \coloneqq a_n \in K^p.
    \end{equation}

\end{enumerate}
\end{defin}

\begin{remark}[Hasse Invariant?]
    The element $\widetilde{[x^{p-1}]}f^{p-1}$ for $D=f\frac{\partial}{\partial x}$ can be considered as a generalization of the Hasse invariant of an elliptic curve to any vector field on any curve, see \cite[Proposition 4.21]{Hartshorne77} and \cite[page 354]{Mazur_Katz_elliptic_curves} for a series of equivalent definitions for elliptic curves.
\end{remark}

\begin{lemma}
    Let $k$ be a perfect field. 

    \begin{enumerate}
        \item Let $f\in k(x)$. Then the inclusion $k(x)\to k((x))$ preserves the value of $\widetilde{[x^{n}]}f$.
        \item Let $f\in k((x))$ and $n\in \{0,1,\ldots,p-1\}$, then if $f=\sum_{n\in \mathbb{Z}} a_n x^n$, then
        \begin{equation}\label{eq: tilde x^n in k((x))}
            \widetilde{[x^{n}]}f = \sum_{\substack{m=pk+n,\\ k\in \mathbb{Z}}} x^{pk}a_m.
        \end{equation}
        Consequently, for any $f,g\in k((x))$ we have $\widetilde{[x^{n}]}f+\widetilde{[x^{n}]}g=\widetilde{[x^{n}]}(f+g)$.
        \item Let $f\in k((x)), g\in k((x))^p$, and $n\in \{0,1,\ldots,p-1\}$, then
        \begin{equation}
        \widetilde{[x^{n}]}(gf)=g \cdot \widetilde{[x^{n}]}f.
        \end{equation}
        In particular, if $f\ne0$, then we have
        \begin{equation}\label{eq: trick fp 1/f}
        \widetilde{[x^{p-1}]}f^{p-1}=f^p \cdot \widetilde{[x^{p-1}]}\frac{1}{f}.
        \end{equation}
    \end{enumerate}
\end{lemma}

\begin{proof} We prove:

    (1) Both $k(x),k((x))$ have the same basis over their $p$-powers, so the coefficients must agree.

    (2) The coefficients in Equation \ref{eq: tilde x^n in k((x))} are $p$-powers that satisfy $f=a_0+a_1x+a_2x^2+\ldots+a_{p-1}x^{p-1}$, and the coefficients of a sum are sums of the coefficients.

    (3) If $f=a_0+a_1x+a_2x^2+\ldots+a_{p-1}x^{p-1}$, then $gf=ga_0+ga_1x+ga_2x^2+\ldots+ga_{p-1}x^{p-1}$. 
    
    And the last part is the above applied to $1/f$ and $f^p$.
\end{proof}

\begin{lemma}\label{lem: n>=2 coeff is zero}
    Let $r\ge 1$ and $n\ge 2$ be integers. Let $p$ be a prime number such that $p>n-1$.
    Then
    \[
    p \text{ divides } \binom{n+(pr-1)-1}{pr-1}={(-1)}^{pr-1}\binom{-n}{pr-1}
    \]
    Moreover, if $p=n-1$, then divisibility no longer holds by Wilson's theorem and its generalizations \cite{Wolstenholme_theorem_history}.
\end{lemma}

\begin{proof}
    The equality with the minus sign follows from the general definition of the Newton symbol for arbitrary numbers.

    The divisibility is elementary. We have $\binom{n+(pr-1)-1}{pr-1}=\frac{pr}{n-1}\binom{n+(pr-1)-1}{n-2}$, since $p>n-1$, thus it is not canceled by any factor from $n-1$.
\end{proof}

The following are the core theorems of this paper.

\begin{thm}[Main Formula]\label{Main Formula}
    Let $k$ be a perfect field. Let $K= k(x)$.
    Let $f\in K$ be a nonzero element.
    Let $PFD(\frac{1}{f})=P+L+H$ be its partial fraction decomposition (over $\overline{k}$), see Theorem \ref{thm: PFD}.

    Let $e$ be the degree of $P$. Let $M$ be the highest order of pole in $H$. Let $p$ be a prime number such that $p-1>e$ and $p>M-1$.

    Let $D=f\frac{\partial}{\partial x}$.
    Then we have that
    \begin{equation}
        D^p= \mu D \hspace{1 cm} \text{, where } \mu=f^p\overline{L}=\frac{\overline{L}}{P^p+L^p+H^p}\in K^p
    \end{equation}
    and $\overline{L}=\sum_{i=1}^d \frac{c_i}{x_i^p-a_i^p}$ if ${L}=\sum_{i=1}^d \frac{c_i}{x_i-a_i}$.
\end{thm}

\begin{proof}
    By Proposition \ref{prop: computing p-powers on curves}, we have that
    \[
    D^p= \mu D \hspace{1 cm} \text{, where } \mu=\widetilde{[x^{p-1}]}f^{p-1}.
    \]  
    By Equation \ref{eq: trick fp 1/f}, we have that
    \[
    \widetilde{[x^{p-1}]}f^{p-1}=f^p \cdot \widetilde{[x^{p-1}]}\frac{1}{f}.
    \]
    Thus we compute $\widetilde{[x^{p-1}]}\frac{1}{f}$.

    Let $PFD(\frac{1}{f})=P+L+H$, see Theorem \ref{thm: PFD}, where
    \begin{enumerate}
        \item $P=a_0+a_1x+\ldots+a_e x^e$,
        \item $L=\sum_{i=1}^d\frac{c_i}{x-\alpha_i}$,
        \item $H=\sum_{i=1}^d\sum_{m>1}\frac{c_{i,m}}{(x-\alpha_i)^m}$.
    \end{enumerate}
    Thus $M$ is the highest $m$ such that there is $i$ such that $c_{i,m}\ne 0$.
    
    We compute the contribution of each part separately:
    \begin{enumerate}
        \item If $p-1>e$, then $\widetilde{[x^{p-1}]}P=0$ by Equation \ref{eq: tilde x^n in k((x))}.
        \item We fix $i$, we assume $\alpha_i\ne 0$ and we compute:
        \[
        \frac{c_i}{x-\alpha_i}=\frac{-\frac{c_i}{\alpha_i}}{1-\frac{x}{\alpha_i}}=-\frac{c_i}{\alpha_i}\sum_{n\ge 0}\left(\frac{x}{\alpha_i}\right)^n=\sum_{n\ge 0}\frac{-c_i}{\alpha_i^{n+1}}x^n
        \]
        so we can conclude, by Equation \ref{eq: tilde x^n in k((x))}, that
        \[
        \widetilde{[x^{p-1}]}\frac{c_i}{x-\alpha_i}=\sum_{k\ge 0}\frac{-c_i}{\alpha_i^{p(k+1)}}x^{pk}=\frac{c_i}{x^p-\alpha_i^p}.
        \]
        For $\alpha_i=0$, we directly get $\widetilde{[x^{p-1}]}(\frac{c_i}{x_i})=\widetilde{[x^{p-1}]}(\frac{c_i}{x_i^p}){x_i}^{p-1}=\frac{c_i}{x_i^p}$.
        
        \item We fix $m,i$, and assume $\alpha_i\ne 0$ and compute:
        \[
        \frac{c_{i,m}}{(x-\alpha_i)^m}=\frac{c_{i,m}}{(-\alpha_i)^m(1-\frac{x}{\alpha_i})^m}=\frac{c_{i,m}}{(-\alpha_i)^m}\sum_{n\ge 0} \binom{-m}{ n}\left(\frac{-x}{\alpha_i^m}\right)^n=\frac{c_{i,m}}{(-\alpha_i)^m}\sum_{n\ge 0} \binom{m+n-1}{ n}\left(\frac{x}{\alpha_i^m}\right)^n.
        \]

        Therefore, for $r\ge 1$, $m\ge 2$, $p>M-1\ge m-1$ by Lemma \ref{lem: n>=2 coeff is zero}, we have 
        \[
        [x^{pr-1}]\frac{c_{i,m}}{(x-\alpha_i)^m}=\frac{c_{i,m}}{(-\alpha_i)^m \alpha_i^{m(pr-1)}}\binom{m+(pr-1)-1}{pr-1}=0.
        \]
        Consequently, by Equation \ref{eq: tilde x^n in k((x))}, we have
        \[
        \widetilde{[x^{p-1}]}\frac{c_{i,m}}{(x-\alpha_i)^m}=\sum_{r\ge 1} x^{p(r-1)}[x^{pr-1}]\frac{c_{i,m}}{(x-\alpha_i)^m}=0.
        \]
        For $\alpha_i=0$, we directly have
        $\widetilde{[x^{p-1}]}\frac{c_{i,m}}{x^m}=\widetilde{[x^{p-1}]}\frac{c_{i,m}}{x^p}x^{p-m}=0$, because $p-1>p-m\ge 0$.
    \end{enumerate}
    We conclude the main formula by joining the above computations.
\end{proof}

\begin{thm}[Main Theorem for $n=1$ and $p>0$]\label{thm: main: n=1, p>0}
    We follow the assumptions and notations from Theorem \ref{Main Formula}. In particular, the prime $p$ is assumed to be big enough. The following are equivalent:
    \begin{enumerate}
        \item $\frac{\overline{L}}{P^p+L^p+H^p}=\mu \in \overline{k}$,
        \item One of the following is true:
        \begin{enumerate}
            \item $L=0$. In this case: $\mu=0$.
            \item $P=H=0$ and there exists $\eta\in \overline{k}$, $\eta \ne 0$ such that each $c_i$ satisfies $(\eta x)^{p}=(\eta x)$. 
            
            In this case: $\mu \ne 0$ and $\eta^{p-1}=\mu$.
        \end{enumerate}
    \end{enumerate}
\end{thm}

\begin{proof}
    (1) implies (2). 
    This is an easy consequence of PFD's uniqueness, Theorem \ref{thm: PFD}. 
    
    Indeed, let $\mu\in \overline{k}$ be the one from (1), then
    \[
    \mu P^p+ (\mu L^p-\overline{L})+ \mu H^p =0
    \]
    This is a PFD decomposition for $0$. The orders in $(\mu L^p-\overline{L})$ and $\mu H^p$ are different, so the uniqueness forces:
    \[
    \mu P^p=0, \hspace{1 cm}  \mu L^p-\overline{L}=0, \hspace{1 cm} \mu H^p=0
    \]
    If $\mu=0$, then $\overline{L}=0$, thus $L=0$.

    If $\mu\ne 0$, then $P^p=H^p=0$, so $P=H=0$, and $\mu L^p-\overline{L}=0$. This last condition means that for every $i$ we have $\mu c_i^p-c_i=0$. We take $\eta^{p-1}=\mu$, then $(\eta c_i)^p=\eta\mu c_i^p=\eta c_i$. 

    (2) implies (1). 
    
    If (a), then all $c_i$ are zero, and thus $\overline{L}=0$, and consequently $\mu=0\in \overline{k}$.

    If (b), then $\mu =\frac{\overline{L}}{L^p}$. Let $\eta \ne 0$, $\eta \in \overline{k}$, be such that all $c_i$ satisfy $(\eta x)^{p}=(\eta x)$, then
    \[
    \mu=\frac{\overline{L}}{L^p}=\frac{\eta^p\overline{L}}{\eta^p L^p}=\eta^{p-1}\frac{\eta\overline{L}}{\eta^{p} L^p}=\eta^{p-1}\in \overline{k}.
    \]
    This finishes the proof.
\end{proof}

\begin{lemma}[What is this $\eta$ for, really?]\label{lemma: n=1, p>0, eta meaning}
    The condition (2b) in Theorem \ref{thm: main: n=1, p>0} can be equivalently restated: 
    the dimension of the vector space $\op{Span}_\mathbb{\mathbb{F}_p}(c_1,c_2,\ldots,c_d)\subset \overline{k}$ is $1$.

    In particular, if $c_i\ne0$, then we can take $\eta=1/c_i$.
\end{lemma}

\begin{proof}
    Indeed, there is $c_i\ne 0$, thus if $\op{dim}=1$, we can put $\eta=\frac{1}{c_i}$, then $\eta c_j\in\mathbb{F}_p$, so $(\eta c_j)^{p}=(\eta c_j)$. And if this $\eta \ne0$ exists, then $\frac{(\eta c_j)^{p}}{(\eta c_i)^{p}}=\left(\frac{c_j}{c_i}\right)^p=\frac{c_j}{c_i}$, so $\frac{c_j}{c_i}\in \mathbb{F}_p$, i.e., $c_i$ is a basis over $\mathbb{F}_p$.
\end{proof}

\begin{thm}[Main Theorem for $n=1$]\label{thm: main: n=1}
    Let $k$ be a field of characteristic zero. 
    Let $f\in k(x)$, $f\ne 0$. 
    Let $PHD(\frac{1}{f})=P+L+H$ with ${L}=\sum_{i=1}^d \frac{c_i}{x_i-a_i}$, see Theorem \ref{thm: PFD}.
    
    Let $D=f\frac{\partial}{\partial x}:k(x) \to k(x)$ be a derivation. 

    Then the following are equivalent:
    \begin{enumerate}
        \item There is a finite type model of $k$ over integers, such that for all all maximal ideals $q$ from an open dense subset of the model, $D$ reduces modulo $q$ and we have $\mu_q\in \overline{k(q)}$, where $D^p=\mu_q D$.
        \item One of the following is true:
        \begin{enumerate}
            \item $L=0$.
            \item $P=H=0$, and the vector space spanned by the residues of $\frac{1}{f}$, the vector space $\op{Span}_\mathbb{Q}(c_1,c_2,\ldots,c_d)\subset \overline{k}$, has dimension $1$ over rational numbers.
        \end{enumerate}
    \end{enumerate}
\end{thm}

\begin{proof} 

    \textbf{(1) implies (2).}

    By Theorem \ref{thm: main: n=1, p>0} and Lemma \ref{lemma: n=1, p>0, eta meaning} we learn that, possibly after enlarging characteristic $p$ of the residue fields $k(q)$ to have $p-1>e,p>M-1$, that our assumption implies, for each $q$ over such $p$ that either: $L=0$, or $P=H=0$ plus extra.

    If both happen infinitely many times, then $1/f$ would be forced to be zero as a sum of two zeros. This is a contradiction.
    So, by enlarging $p$ a bit more, we can assume that for all $q$ above $p$, only one of these cases holds.

    If $L=0$ holds modulo all our $q$, then $L=0$ before reductions.

    If $P=H=0$, then we get $L\ne 0$. Then there exists $c_i\ne 0$. Then for all $j$ we have $\left(\frac{c_j}{c_i}\right)^p=\frac{c_j}{c_i}$ modulo all our $q$, so by Kronecker Theorem ++ \ref{thm: kronecker ++}, we have $\frac{c_j}{c_i}\in \mathbb{Q}$, i.e., the vector space is of dimension $1$.

    \textbf{(2) implies (1).} 
    
    This is Theorem \ref{thm: main: n=1, p>0} (2) implies (1) and Lemma \ref{lemma: n=1, p>0, eta meaning}, a delegated direct check.
\end{proof}

\begin{defin}\label{def: constant-type}
Let $k$ be a field of characteristic zero.
    We say that $f\in k(X)$ with $f\ne 0$ is of \textbf{constant-type} if it satisfies the equivalent condition of Theorem \ref{thm: main: n=1}.
\end{defin}

\section{Algebraic Integrability}

\subsection{For the vector field}

We use the case $n=1$, Theorem \ref{thm: main: n=1}, to prove the following.

\begin{thm}[Main Theorem $n\ge 2$]\label{thm: main theorem n>=2}
    Let $k=\overline{\mathbb{Q}}$ be an algebraic closure of rational numbers.

Let $n\ge 2$ be an integer.
Let $x_1,x_2,\ldots, x_n$ be free variables.

Let $f_1,f_2,\ldots,f_n \in \overline{\mathbb{Q}}(X)$ be non-zero rational functions in one variable.

For each $i=1,2,\ldots,n$, let $PFD(\frac{1}{f_i})=P_i+L_i+H_i$, see Theorem \ref{thm: PFD}, with $L_i=\sum_{j=1}^{d_i}\frac{c_{i,j}}{X-\alpha_{i,j}}$. We define $V_i=\op{Span}_\mathbb{Q}(c_{i,1},c_{i,2},\ldots,c_{i,d_i})\subset \overline{\mathbb{Q}}$: the vector space spanned by residues of $\frac{1}{f_i}$.

Let $D=
f_1(x_1)\frac{\partial}{\partial x_1}+f_2(x_2)\frac{\partial}{\partial x_2}+\ldots+f_n(x_n)\frac{\partial}{\partial x_n}$ be a derivation on $\overline{\mathbb{Q}}(x_1,\ldots,x_n)$.

Then the following are equivalent:
\begin{enumerate}
    \item Let $K$ be a number field such that $f_i \in K(X)$. Let $\cO_K$ be the ring of integers of $K$.
    
    For almost all prime numbers $p$, and all prime ideals $q$ in $\cO_K$ over these $p$, the derivation $D$ is $p$-closed modulo $q$.
    \item One of the following is true:
    \begin{enumerate}
        \item For all $i$, we have $L_i=0$.
        \item For all $i$, we have $P_i=H_i=0$ and the dimension of $V_i$ is $1$. 
        
        And the dimension of $V_1+V_2+\ldots+V_n$ is $1$.
    \end{enumerate}
    \item The foliation $\F$ defined by $D$, i.e., $\F=\overline{\mathbb{Q}}(x_1,x_2,\ldots,x_n)D$, is algebraically integrable, i.e., the subfield
    $W=\op{Ann}(\F)=\{g\in \overline{\mathbb{Q}}(x_1,x_2,\ldots,x_n): D(g)=0\}$ is of transcendental degree $n-1$, and $\F=\op{Der}_W(\overline{\mathbb{Q}}(x_1,x_2,\ldots,x_n))$.
\end{enumerate}

Moreover, we can write down explicit transcendental bases:
\begin{itemize}
    \item (2a): Let $1<j\le n$. Then we can take:
        \begin{equation}
            F_{1j}\coloneqq \int \frac{dx_1}{f_1}-\int \frac{dx_j}{f_j} \in \overline{\mathbb{Q}}(x_1,x_j)\setminus \left(\overline{\mathbb{Q}}(x_1)\cup \overline{\mathbb{Q}}(x_j)\right).
        \end{equation}

    \item (2b): There exists $\eta \in \overline{\mathbb{Q}}$, $\eta \ne 0$, and an integer $N>0$ such that $N\eta c_{i,j}\in \mathbb{Z}$ for all $i,j$. Now, let $1<j\le n$. Then we can take:
    \begin{equation}
        G_{1j} \coloneqq \op{exp}\left(N\eta \left(\frac{dx_1}{f_1}-\int \frac{dx_j}{f_j}\right)\right)\in \overline{\mathbb{Q}}(x_1,x_j)\setminus \left(\overline{\mathbb{Q}}(x_1)\cup \overline{\mathbb{Q}}(x_j)\right).
    \end{equation}
\end{itemize}
    Thus we always have that $W$ is a separable closure of $\overline{\mathbb{Q}}(F_{ij})$ or $\overline{\mathbb{Q}}(G_{ij})$. 
\end{thm}

\begin{remark}[It works over any field.]
    Theorem \ref{thm: main theorem n>=2} is true over any field $k$ of characteristic zero, not only over $\overline{\mathbb{Q}}$. Indeed, our main inputs, Kronecker's theorem ++ \ref{thm: kronecker ++} and Proposition \ref{prop: easy direction AlgIntConj}, are over any such field $k$, and our explicit computations culminating in Theorem \ref{thm: main: n=1} are also over any such $k$. The general statement requires only changes to condition (1) to mention an open subset of a more abstract model; the rest is a mechanical substitution. Nothing deep here.

    Nevertheless, I believe that not putting abstract $k$ into the statement helps with its digestion.
\end{remark}

\begin{proof} (Of Theorem \ref{thm: main theorem n>=2})

    \textbf{(3) implies (1).}

    This is the easy direction of the algebraic integrability conjecture, Proposition \ref{prop: easy direction AlgIntConj}.

    \textbf{(1) is equivalent to (2)}
    
    This follows from our work for $n=1$. 
    
    First, we observe that for each $i\ne j$ we have $[f_i(x_1)\frac{\partial}{\partial x_i},f_j(x_j)\frac{\partial}{\partial x_j}]=0$. Therefore, by Proposition \ref{prop: identities for p-lie} and Proposition \ref{prop: computing p-powers on curves}, we have that
    \begin{align*}
    D^p&=\left(f_1(x_1)\frac{\partial}{\partial x_1}+f_2(x_2)\frac{\partial}{\partial x_2}+\ldots+f_n(x_n)\frac{\partial}{\partial x_n}\right)^p  \\
    &=\left(f_1(x_1)\frac{\partial}{\partial x_1}\right)^p+\left(f_2(x_2)\frac{\partial}{\partial x_2}\right)^p+\ldots+\left(f_n(x_n)\frac{\partial}{\partial x_n}\right)^p\\
    &=\mu_1f_1(x_1)\frac{\partial}{\partial x_1}+\mu_2f_2(x_2)\frac{\partial}{\partial x_2}+\ldots+\mu_nf_n(x_n)\frac{\partial}{\partial x_n},
    \end{align*}
    where $\mu_i\in \overline{\mathbb{F}_p}(x_i^p)$. Consequently, $D$ is $p$-closed if and only if $\mu_1=\mu_2=\ldots=\mu_n$. In particular, this forces all of them to belong to $\overline{\mathbb{F}_p}$. This means all $f_i$ are constant-type, Definition \ref{def: constant-type}.

    Therefore, by Theorem \ref{thm: main: n=1}, we learn that for each $i$ we have either $L_i\ne 0$ and $\mu_i\ne 0$ for almost all $p$, or $L_i=0$ and $\mu_i=0$ for almost all $i$.

    We cannot have that some $i$ have $\mu_i=0$ and some $\mu_i\ne 0$, thus all of them must be of the same type: 
    for all $i$ we have $L_i=0$, xor for all $i$ we have $L_i\ne 0$.

    The first case is precisely (2a) and the second is the first part of (2b). We will prove the rest of part (2b). We already know $\op{dim}V_i=1$. However, we know that $\mu_i=\mu_j$ for all $i,j$. Let $c_i$ be a nonzero residue of $1/f_i$ and $c_j$ for $1/f_j$, then, by Lemma \ref{lemma: n=1, p>0, eta meaning}, $1/c_i^{p-1}=\mu_i=\mu_j=1/c_j^{p-1}$, i.e., $\left(c_i/c_j\right)^{p}=c_i/c_j$ for almost all $q$, so by Theorem \ref{thm: kronecker ++}, we get $c_i/c_j\in \mathbb{Q}$. Consequently, the dimension of $\sum_{i=1}^n V_i$ is one.

    The same formulas taken the other way show that being $p$-closed is equivalent to (2a) or (2b) holding.

    \textbf{(2) implies (3)}
    
    We use our assumptions to produce explicit elements killed by $D$, and among them there is a transcendental basis.

    \begin{itemize}
        \item Let us assume (2a) is true. We have all $L_i=0$. So $\int \frac{dx_i}{f_i}$ does not include any logarithms. It is a rational function.
        Let $i\ne j$. Then we define
        \[
        F_{ij}\coloneqq \int \frac{dx_i}{f_i}-\int \frac{dx_j}{f_j} \in \overline{\mathbb{Q}}(x_i,x_j)\setminus \left(\overline{\mathbb{Q}}(x_i)\cup \overline{\mathbb{Q}}(x_j)\right).
        \]
        We have that $D(F_{ij})=0$, and the $ F_{1j} $ have transcendental degree $n-1$.

        In this case $W$ is the separable closure of $\overline{\mathbb{Q}}(F_{ij})$.
        \item Let us assume (2b) is true. We have $ P_i=H_i=0$ for all $ i$, so there are only logarithmic terms in $\int \frac{dx_i}{f_i}=\sum_{j=1}^{d_i} c_{i,j}\op{log}(x-\alpha_{i,j})$. (We ignore ``$+ C$'' by choosing $C=0$.)

        Since $\op{dim}\sum_{i=1}^N V_i=1$, there exists $\eta \in \overline{\mathbb{Q}}$, $\eta \ne 0$ such that $\eta c_{i,j}\in {\mathbb{Q}}$ for all $i,j$.

        Let $N>0$ be an integer such that $N\eta c_{i,j}\in \mathbb{Z}$. The set of indices is finite, so we can do that.
        Let $i\ne j$. Then we define
        \[
        G_{ij} \coloneqq \op{exp}\left(N\eta \left(\frac{dx_i}{f_i}-\int \frac{dx_j}{f_j}\right)\right)\in \overline{\mathbb{Q}}(x_i,x_j)\setminus \left(\overline{\mathbb{Q}}(x_i)\cup \overline{\mathbb{Q}}(x_j)\right).
        \]
        These are rational functions such that $D(G_{ij})=0$ which define a subfield of transcendence degree $n-1$. In this case $W$ is the separable closure of $\overline{\mathbb{Q}}(G_{ij})$.
    \end{itemize}

    The theorem is proved.
\end{proof}

\begin{remark}
    I am not sure if the equalities $\overline{\mathbb{Q}}(G_{ij})=W$ or $\overline{\mathbb{Q}}(F_{ij})=W$ hold in general. 
\end{remark}

\subsection{For m-saturations}
The methods of this paper naturally suggest the following follow-up.

\begin{defin}
    Let $\F$ be a foliation over a finite extension $K$ of $\overline{\mathbb{Q}}(x_1,\ldots,x_n)$. We define a \textbf{m-saturation} of $\F$ to be a set 
    \[
    \overline{\F}\coloneqq \{D\in \op{Der}_{\overline{\mathbb{Q}}}(K): \text{$D$ belongs to $\F+\F^p+\F^{p^2}+\ldots$ modulo all $q$ over almost all $p$}\}.
    \]
\end{defin}

\begin{remark}
    The name ``m-saturation'' is a short form for ``modulo-almost-all-primes-saturation''. 

    It is a foliation. First, it is a vector subspace of derivations. Indeed, if $D, E$ belong to $p$-powers of $\F$ for almost all $p$, then their linear combinations over $K$ do the same. The same is true for Lie brackets, because being closed under $p$-powers implies being closed under Lie brackets - this is the main result of \cite{p-powers-give-lie-brackets}.
\end{remark}

The definition of m-saturation was formulated to capture the following example. We do not know if it is of any further use. To judge it properly, one would require a more diverse set of examples and some structural propositions about this notion.

\begin{example}[Computing a m-saturation.]\label{Ex: Computing a m-saturation}
Let $D=
f_1(x_1)\frac{\partial}{\partial x_1}+f_2(x_2)\frac{\partial}{\partial x_2}+\ldots+f_n(x_n)\frac{\partial}{\partial x_n}$ be a derivation on $\overline{\mathbb{Q}}(x_1,\ldots,x_n)$, where $f_1,f_2,\ldots,f_n \in \overline{\mathbb{Q}}(X)$, possibly some of them are zero.

We compute m-saturation for the foliation defined by $D$. 

We simply take $D,D^p,D^{p^2},\ldots$ and perform a linear algebra elimination.
This boils down to rearranging the elements $f_i$ into a few types.

\begin{itemize}
    \item We assume for $i\le m_1$ we have $f_i\ne 0$ and for $i>m_1$ we have $f_i=0$.

    For each $i=1,2,\ldots,m$, let $PFD(\frac{1}{f_i})=P_i+L_i+H_i$, see Theorem \ref{thm: PFD}, with $L_i=\sum_{j=1}^{d_i}\frac{c_{i,j}}{X-\alpha_j}$. We define $V_i=\op{Span}_\mathbb{Q}(c_{i,1},c_{i,2},\ldots,c_{i,d_i})\subset \overline{\mathbb{Q}}$: the vector space spanned by residues of $\frac{1}{f_i}$.
    \item We assume for $m_2\le m_1$, for $i\le m_2$ we have that $f_i$ is of constant-type, Definition \ref{def: constant-type}. And for $m_1\ge i>m_2$ it is not.
    \item We assume for $m_3\le m_2$, for $i\le m_3$ we have $P_i=H_i=0$, and for $m_2\ge i>m_3$ we have $L_i=0$.
    \item Finally, we group $i\le m_3$ into maximal subsets $S_1,S_2,\ldots,S_k\subset \{1,2,\ldots,m_3\}$ such that for each $j$: the dimension of $\sum_{i\in S_j} V_i$ is one. These subsets are disjoint, and their union equals $\{1,2,\ldots,m_3\}$.
\end{itemize}

We define the following elements:
\begin{itemize}
    \item For each $i\in\{m_2+1,\ldots,m_1\}$ we define $A_i=f_i(x_i)\frac{\partial}{\partial x_i}$ .
    \item $B=\sum_{i=m_3+1}^{m_2}f_i(x_i)\frac{\partial}{\partial x_i}$.
    \item For each $j\in \{1,\ldots,k\}$ we define $C_j=\sum_{i\in S_j}f_i(x_i)\frac{\partial}{\partial x_i}$.
\end{itemize}

We claim that the m-saturation $\overline{\F}$ of $\F$ is the foliation defined by $\{A_i,B,C_j\}$.

Indeed, this is because these elements commute with each other and their $p$-powers are of particular types:
\begin{itemize}
    \item $A_i^p = \mu_i A_i$, where $\mu_i$ is a nonconstant function in $x_i$.
    \item $B^p=0$.
    \item $C_j^p=\mu_j C_j$, where $\mu \in \overline{\mathbb{F}_p}$, but these $\mu_j$ for different $j$ are linearly independent.
\end{itemize}
This finishes the computation.
\end{example}

\begin{cor}[The actual answer to Problem \ref{problem1} for arbitrary $n\ge 2$]\label{cor: the actual answer}
    The foliation defined by $D$ from Equation \ref{eq: formula - definition} admits a non-constant rational function $F$ such that $D(F)=0$ if and only if the m-saturation $\overline{\F}$ of the foliation $\F$ defined by $D$ is of rank $\le n-1$, i.e., it is not everything.
\end{cor}

\begin{proof}
    If $D$ admits $F$, then the rank of $\overline{\F}$ modulo $p$ is $\le n-1$ for almost all $p$, so it is not $n$.

    If the rank of $\overline{\F}$ is $\le n-1$, then we can check that some $F$ exists by cases, because we have an explicit form for this m-saturation, Example \ref{Ex: Computing a m-saturation}.

    If $m_1<n$, then $F=x_n$ works.

    If $m_1=n$, then the rank is $\le n-1$ if and only if $m_2-m_3\ge 2$, or there is $j$ such that $|S_j|\ge 2$.
    
    So, if $m_2-m_3\ge 2$, then $F=\int \frac{dx_{m_2}}{f_{m_2}}-\int \frac{dx_{m_2-1}}{f_{m_2-1}}$ works.

    And, if $|S_j|\ge 2$, then $F=\op{exp}\left(N\eta(\int \frac{dx_{a}}{f_{a}}-\int \frac{dx_{b}}{f_{b}})\right)$ for $a,b\in S_j,a\ne b$, an algebraic number $\eta \ne 0$ and an integer $N\ne 0$.
\end{proof}

\section{Explicit Algebraic Integrability over Rational Numbers}
The main theorem \ref{thm: main theorem n>=2} is wonderful, but how hard is it to actually use it? In this section, we show that the cases $n=2$ and $n\ge 2$ are equivalent, and how to obtain the explicit formulas for $n=2$. All of this boils down to a careful study of vector spaces and fields spanned by residues of rational functions in the case of $P=H=0$.

\begin{lemma}[The case $n=2$ is the case $n\ge 2$]
    Let $k$ be a field of characteristic zero.
    Let $D$ be a derivation of the form $D= \sum_{i=1}^n f_i(x_i)\frac{\partial}{\partial x_i}$ where $f_i\in k(X)$ are nonzero rational functions.

    Then, $D$ defines an algebraically integrable foliation on $k(x_1,x_2,\ldots,x_n,\ldots)$ if and only if 
    for each $1\le i<j\le n$ the derivation $E_{ij}=f_i(x_i)\frac{\partial}{\partial x_i}+f_j(x_j)\frac{\partial}{\partial x_j}$ defines an algebraically integrable foliation on $k(x_i,x_j)$.
\end{lemma}

\begin{proof}
    We use the main theorem \ref{thm: main theorem n>=2}, the conditions (2) and (3).
    
    First, a necessary condition is that all $f_i$ are constant-type, Definition \ref{def: constant-type}, and that they are compatible: all $1/f_i$ are $P+H$, or all are $L$ and extra. 
    
    In the first case, we have an exact equivalence between conditions for $D$ and $E_{ij}$.

    In the second case, we need to observe that $V_1+\ldots+V_n$ has dimension $1$ if and only if, for each $i,j$, $V_i+V_j$ has dimension $1$.
\end{proof}

\subsection{Obtaining Explicit Formulas} We start with some trivia on residues.

\begin{lemma}[Computing Residues]\label{lem: residues via lhopital}
    Let $k$ be a field of characteristic zero.
    Let $f\in k[X]$ be a separable polynomial. Let $g\in k[X]$ be a polynomial. Let $\alpha\in \overline{k}$ be such that $f(\alpha)=0$ and $g(\alpha)\ne0$; then the residue of $g/f$ at $\alpha$ is equal to $g(\alpha)/f'(\alpha)$.
\end{lemma}

\begin{proof}
Let $c$ be the residue; by elementary calculus, if $k\subset\mathbb{C}$:
\[
c=\lim_{x\to\alpha} \frac{(x-\alpha)g(x)}{f(x)}=\lim_{x\to\alpha} \frac{g(x)+(x-a)g'(x)}{f'(x)}=\frac{g(\alpha)}{f'(\alpha)}.
\]
This can also be obtained purely algebraically by a direct computation over any field.
\end{proof}

\begin{lemma}[Lagrange Approximation]\label{lem: Lagrange Approximation}
    Let $\alpha_1,\ldots,\alpha_d$ be distinct elements of a field $k$. Let $f(x)=c\prod_{i=1}^d(x-\alpha_i)$, where $c\in k$, $c\ne 0$. Let $g\in k[x]$ be a polynomial of degree $<d$, then
    \[
    g(x)=\sum_{i=1}^d g(\alpha_i)\frac{f(x)}{(x-\alpha_i)f'(\alpha_i)}.
    \]
    In particular, we have that 
    $
    [x^{d-1}]g(x)=c\sum_{i=1}^d \frac{g(\alpha_i)}{f'(\alpha_i)}
    $ holds. 
\end{lemma}

\begin{proof}
    Any polynomial of degree $d-1$ is determined by its values at $d$ different arguments, because of the Vandermonde matrix.
\end{proof}

We go towards the explicit formulas. We start with a lemma-corollary.

\begin{cor}\label{cor: degree no more than 2}
    Let $k$ be a field of characteristic zero.
    Let $f,g\in k(X)$ be nonzero rational functions such that $PFD(1/f)=L(1/f)$ and $PFD(1/g)=L(1/g)$. If the foliation defined by
    $D= f(x)\frac{\partial}{\partial x} +g(y)\frac{\partial}{\partial y}$ is algebraically integrable, then the field $L$ generated by residues of $1/f,1/g$ over rational numbers satisfies $(L:k)\le 2$.
\end{cor}

\begin{proof}
    This is {\cite[a comment at page 6 that is to follow from Proposition 7 therein]{rational_first_for_separated_2025}}.
    
    We reprove it using elementary Galois theory; see, e.g., \cite[Chapter 8]{Jacobson-BasicAlgebraII}. Our proof differs from the one in \cite{rational_first_for_separated_2025}, or the original source they cite, without indicating where it is stated and proved, which I could not identify. Their statements are also unclear to me; they are quite technical without adding much.

    By the main theorem \ref{thm: main theorem n>=2}, the conditions (2) and (3), $D$ defines an algebraically integrable foliation if and only if the spaces $V_f, V_g$ spanned by residues of $1/f,1/g$ are of dimension $1$ each, and $V_f+V_g$ is of dimension $1$ too. Therefore, the field generated by residues of $1/f$ and $1/g$ is the same one as the one generated by the residues of $1/f$ alone. The same is true for $1/g$. So, we can focus on $V_f$ alone.

    We have that $PFD(1/f)=L(1/f)$, so
    $
    1/f= \sum_{i=1}^r \frac{h_i(x)}{f_i(x)},
    $
    where $f_i$ are irreducible rational polynomials coprime to each other, and $h_i\ne 0$ are polynomials with $\op{deg}(h_i)<\op{deg}(f_i)$.
    Therefore, $V_f$ is spanned by residues of $\frac{h_i(x)}{f_i(x)}$ for any $i$.

    We fix $i$ and rename $h=h_i, f=f_i$. (We have a notational overload, but it does not matter.)
    
    Let $\alpha_i$ for $i=1,2,\ldots,d$ be poles of $h/f$. And $\gamma_i$ residues at them respectively, so  we have $\gamma_i=\frac{h(\alpha_i)}{f'(\alpha_i)}$ by Lemma \ref{lem: residues via lhopital}.
    By our assumptions, $V_f=\op{span}_\mathbb{Q}(\gamma_1,\ldots,\gamma_d)$ has dimension $1$. And each of $\gamma_i$ is nonzero. So, there are nonzero rational numbers $c_i \in \mathbb{Q}$ such that $\gamma_i=c_i\gamma_1$.
    
    Let $G$ be the Galois group of the splitting field of $f$. $G$ acts transitively on $\alpha_i$, and thus it acts transitively on $\gamma_i$.
    Indeed, if $\sigma\in G$ does $\sigma(\alpha_1)=\alpha_i$, then
    $\sigma(\gamma_1)=\gamma_i$, because $\gamma_i=\frac{h(\alpha_i)}{f'(\alpha_i)}$, and $h,f$ are rational polynomials over $k$.

    Let $e$ be the order of $\sigma$, then
    \begin{equation}\label{eq: argument from order at most 2}
        \gamma_1=\sigma^e (\gamma_1)=\sigma^{e-1}(c_i\gamma_1)=c_i\sigma^{e-1}(\gamma_1)=\ldots=c_i^e\gamma_1.
    \end{equation}

    Consequently, $c_i^e=1$, so $c_i=\pm 1$, because these are the only rational roots of unity. 
    
    If all $c_i$ are $1$, then $\gamma_1$ is fixed by $G$, and thus rational, i.e., the field generated by all $\gamma_i$ is $k$.

    If some $c_i$ are $-1$, then $\gamma_1,\gamma_i=-\gamma_1$ is a Galois orbit of $\gamma_1$, and thus $k(\gamma_1,\gamma_2,\ldots,\gamma_d)=k(\gamma_1,-\gamma_1)=k(\gamma_1)$ is of degree $2$.

    This finishes the proof.
\end{proof}

\begin{problem}[Algebraic Number Theory meets Vector Fields]
    For arbitrary derivation $D$ on a field over a number field with finitely many variables. Are there more complex restrictions similar to Corollary \ref{cor: degree no more than 2} that we could exploit towards the general algebraic integrability conjecture \ref{conj: alg int conj}?
\end{problem}

    Corollary \ref{cor: degree no more than 2} inspires us to study the following question:
    Which rational functions $f\in\mathbb{Q}(x)$ satisfying $PFD(1/f)=L(1/f)$
    have a field generated by residues of $1/f$ of degree $\le 2$?
    The surprise is that this can be answered completely.

\begin{lemma}[Low Orders of Residue Subfields imply Explicit Numerators]\label{lem: explicit h for residues field of orders 1,2}
    Let $k$ be a field. Let $f\in k[x]$ be a separable irreducible polynomial over $k$ of degree $d$ with roots $a_1,\ldots,a_d$.  Let $h\in k[x]$ be a polynomial of degree $<d$.
    Then the field $L\coloneqq k\left(\frac{h(a_1)}{f'(a_1)},\frac{h(a_2)}{f'(a_2)} ,\frac{h(a_3)}{f'(a_3)}, \ldots, \frac{h(a_d)}{f'(a_d))}\right)$
    satisfies $(L:k) \le 2$ if and only if:
    \begin{enumerate}
        \item $(L:k) =1$, and $h=cf'$ for some $c\in k$.
        \item $(L:k) = 2$, and $f=g\cdot \overline{g}$ factorizes over $L$ into two conjugate factors, 
        \\and there is $c\in L\setminus k$ such that $h=cg' \overline{g}+ \overline{c}g\overline{g}'$.
    \end{enumerate}
\end{lemma}

\begin{proof}
    Let $F$ be the splitting field of $f$, i.e., $F=k(a_1,\ldots,a_d)$. The extension $F/k$ is Galois, because every splitting field of a separable polynomial is Galois. Let $G$ denote $\op{Gal}(F/k)$.

    $G$ acts transitively on all $a_i$, and thus on all $\frac{h(a_i)}{f(a_i)}$. Indeed, if $\sigma\in G$ satisfies $\sigma(a_1)=a_i$, then $\sigma(\frac{h(a_1)}{f(a_1)})=\frac{h(a_i)}{f(a_i)}$, because $h,f$ are polynomials over $k$.

    We put $\gamma_i=\frac{h(a_1)}{f(a_1)}$.
    
    We assume $(L:k)\le 2$. Thus, the Galois orbit of $\gamma_1\in L$ is the set $A=\{\gamma_1,\ldots,\gamma_d\}$, but, since the order of this extension is $\le 2$, this set is of size $1$, or $2$.

    If $|A|=1$, then for all $i$ we have $\gamma_1=\gamma_i$, thus $\gamma_1\in k$. By Lagrange approximation, Lemma \ref{lem: Lagrange Approximation}, applied to $h$ and $f'$, we get
    \begin{align*}
     h(x)&=\sum_{i=1}^d h(a_i)\frac{f(x)}{(x-a_i)f'(a_i)}=\sum_{i=1}^d \gamma_i\frac{f(x)}{(x-a_i)}=\\
     &=\gamma_1\sum_{i=1}^d \frac{f(x)}{(x-a_i)}=\gamma_1\sum_{i=1}^d f'(a_i)\frac{f(x)}{(x-a_i)f'(a_i)}=\gamma_1f'(x).
    \end{align*}
    We can take $c=\gamma_1$.

    If $|A|=2$, then the set $\{a_1,\ldots,a_d\}$ splits into two sets $A_1=\{a_i: \gamma_i=\gamma_1\}$ and $A_2=\{a_i: \gamma_i\ne\gamma_1\}$. By transitivity, these sets are of equal size $d/2$, thus $2|d$ and $d\ge 2$. We put $d=2d'$. We denote the roots in $A_1$ by $\alpha_1,\ldots,\alpha_{d'}$ and the ones in $A_2$ by $\beta_1,\ldots,\beta_{d'}$. We assume, without any loss of generality, that $a_2\in A_2$. So, we will write $\gamma_2$ for the second value.
    
    We put $g(x)=\prod_{i=1}^{d'}(x-\alpha_i)$ and $\overline{g}(x)=\prod_{i=1}^{d'}(x-\beta_i)$, so $f(x)=g(x)\overline{g}(x)$, if $f$ is monic. We observe that the subgroup of $G$ fixing $L$ fixes $g$ and $\overline{g}$, so they are polynomials over $L$ that are conjugate to each other with respect to $G$.

    The polynomial $H(x)=\gamma_1 g'(x)\overline{g}(x)+\gamma_2 g(x)\overline{g}'(x)-h(x)$ is by the definitions of degree $<d$ and zero at all roots of $f$, thus $H=0$, i.e., we can put $c=\gamma_1$ and we are done.

    A simple direct check proves the other implication. 
    This finishes the proof.
\end{proof}

\begin{remark}[Higher Degrees?]
    I believe that one can produce explicit formulas for numerators for $(L:k)>2$ too. These will be more complicated, and we do not need them, so we do not.
\end{remark}

\begin{thm}\label{thm: explicit, n=2, both L, over rational number}
    This is a proof-pointer for Theorem \ref{thm: explicit main thm in intro}.
\end{thm}

\begin{proof} We assume (1). (Recall that condition (1) means that $D$ defines an algebraically integrable foliation.) We use Theorem \ref{thm: main theorem n>=2} applied to $n=2,x_1=x,x_2=y,f_1=f,f_2=g$, and we use the notation from therein. We open the cases of Theorem \ref{thm: main theorem n>=2} (3) implies (2) and specialize the answer with Lemma \ref{lem: explicit h for residues field of orders 1,2}.

If $L_1=L_2=0$, then we have the case where all residues are zero.

If $P_1=H_1=P_2=H_2=0$, then we can use Lemma \ref{lem: explicit h for residues field of orders 1,2}. 

If the field $L$ generated by the residues of $1/f,1/g$ is $\mathbb{Q}$, then we get the case (2b). Indeed, we apply the lemma to every summand of the PFD of $1/f,1/g$, Theorem \ref{thm: PFD}, and that's enough.

If the field $L$ generated by the residues of $1/f,1/g$ is a proper quadratic extension, then, in the same way as above, we get almost the case (2c). We miss the condition that all the $c$ we get shall satisfy $c+\overline{c}=0$ from the formulas for the numerators $h=cg' \overline{g}+ \overline{c}g\overline{g}'$. 
This follows from the condition that $V_f+V_g$ is of dimension $1$ over $\mathbb{Q}$. Indeed, the residues in this case are all these $c,\overline{c}$. For each pair, we get that $c=e\overline{c}$, where $e\in\mathbb{Q}$, so $e=-1$.  
This finishes the proof of (2).

We assume (2). The case (1) follows by a direct computation of (2) from Theorem \ref{thm: main theorem n>=2}, and then using that (2) implies (3) there. That (3) is our (1), so we are done.

Along the way, (2a) and (2b) to (1) will be explicit, but (2c) to (1) requires the following observation $\op{dim}(V_f+V_g)=1$ is satisfied because, if $L=\mathbb{Q}(\sqrt{q})$, then we get: $c\in \mathbb{Q}\sqrt{q}$, since $c+\overline{c}=1$, so it is of dimension $1$, and Theorem \ref{thm: main theorem n>=2} applies.

This finishes the proof.
\end{proof}

\begin{remark}
    With a bit more control, the proof of Theorem \ref{thm: explicit, n=2, both L, over rational number} can be adapted to work over any $k$, e.g., we can replace the rational $c$ next to $cf'/f$ with an element from the quadratic extension $L$ with $c+\overline{c}=0$ and other similar actions.
\end{remark}

\section{More Examples}
We present some examples and direct applications of our main theorems \ref{thm: main thm in intro}, \ref{thm: explicit main thm in intro}.

\subsubsection{Tori Invariance} Here is a rather surprising general corollary from the main theorem.

\begin{lemma}[Rational Tori Invariance]\label{lem: toric invariance}
    Let $D$ be a derivation of the form $D= \sum_{i=1}^n f_i(x_i)\frac{\partial}{\partial x_i}$ where $f_i\in \overline{\mathbb{Q}}(X)$ are nonzero rational functions, let $A_1,\ldots,A_n\in \mathbb{Q}$ be nonzero rational numbers. Then $D$ defines an algebraically integrable foliation if and only if
    $E=\sum_{i=1}^n A_if_i(x_i)\frac{\partial}{\partial x_i}$ does too.
\end{lemma}

\begin{proof}
    By the main theorem \ref{thm: main theorem n>=2}, we know that $D$ is algebraically integrable if and only if condition (2) holds, which reads: 
    \begin{itemize}
        \item all $f_i$ must be of constant type. This is preserved by multiplying by constants.
        \item they must be compatible, either all $1/f_i$ have $L=0$ or all have $P=H=0$ and extra.
        \begin{itemize}
            \item If all have $L=0$, then this is preserved.
            \item If all have $P=H=0$, then this is preserved. Finally, the extra means that the residues of all $1/f_i$ form $1$ dimensional vector space over rational numbers. By multiplying by $A_i$, we only multiply these residues by $A_i$, so we do not change that dimension.
        \end{itemize}
    \end{itemize}
    This finishes the proof.
\end{proof}

\subsubsection{Ramen Lemma}

\begin{lemma}[Ramen Lemma]\label{lem: ramen lemma}
    Let $k$ be a field. Let $f\in k[x]$ be a separable irreducible polynomial over $k$ of degree $d>2$ with roots $a_1,\ldots,a_d$.
    Then the field $K\coloneqq k\left(f'(a_1),f'(a_2),\ldots,f'(a_d)\right)$
    satisfies $(K:k) > 2$.
\end{lemma}

\begin{proof}
    We go by contradiction. Let us assume that $(K:k) \le 2$.
    
    We use Lemma \ref{lem: explicit h for residues field of orders 1,2} for the constant polynomial $h(x)=1$. 
    
    If $(K:k)=1$, then $1=cf'$ for some $c\in k$. This cannot be true, for degree reasons. 
    
    If $(K:k)=2$, then $1=cg' \overline{g}+ \overline{c}g\overline{g}'$, where $g$ can be taken to be a factor of $f=g\overline{g}$ over $K$ of degree $d/2$. In particular, $d$ is even, and thus $d\ge 4$. By comparing the coefficients next to $x^{d-1}$ we get $c+\overline{c}=0$, and after taking derivatives we get
    $
    g'' \overline{g}+g' \overline{g}'-g'\overline{g}'-g\overline{g}''=0.
    $, i.e., $g'' \overline{g}=g\overline{g}''$. However, $g,\overline{g}$ are coprime (they don't share any roots), so $g$ divides $g''$. This can be true only for $g''=0$, but $d>3$, so $g$ is not linear; therefore, we have a contradiction.
\end{proof}

\begin{remark}[Why ramen?]
    I like ramen. One of my favorite places is Menya Musashi in Kyiv; the one close to Palac Sportu Metro Station. You can get a ramen XXL there. Yummy!

    At that location, 
    with Asem Abdelraouf, we solved the Ramen lemma over dinner, in two different ways.

    I believe that for $k=\mathbb{Q}$, $h=1$, and $d\ge 3$ the field $K$ usually is equal to $L=k(a_1,\ldots,a_d)$, the splitting field of $f$, but this does not always happen. Indeed, for example, the polynomial $f(x)=x^6-2x^3+2$ have $(L:k)=12$, but $(K:k)=6$. So, $6>2$, i.e., the ramen lemma \ref{lem: ramen lemma} works, but $12\ne 6$. This example is due to Asem, and it can be verified with PARI or any other computational software. Initially, I believed that $L=K$ always holds. It does not.
\end{remark}

\begin{example}[$n=2$, Irreducible Monic Polynomial Coefficients]\label{Ex: Irreducible Monic Polynomial Coefficients}
    Let $f,g\in \mathbb{Q}[X]$ be two irreducible monic polynomials. 
    We give all such pairs such that $D=f(x)\frac{\partial}{\partial x}+g(y)\frac{\partial}{\partial y}$ is algebraically integrable. There are only two options:

    \begin{itemize}
        \item Both $f,g$ are linear polynomials.
        \item Both $f,g$ are quadratic polynomials with the same splitting field.
    \end{itemize}
    
    Proof. By Theorem \ref{thm: explicit, n=2, both L, over rational number}, $D$ is algebraically integrable if and only if we have that 
    
    \[
    \frac{1}{f}=\frac{cf'}{f}, \hspace{1cm} \text{or} \hspace{1cm} \frac{1}{f}=\frac{c(f_1'f_2-f_1f_2')}{f},
    \]
    where $c$ is a rational number in the first case, and a quadratic one in the second such that $c+\overline{c}=0$, and $f_1f_2=f$ over $L=\mathbb{Q}(c)$ are two conjugate factors of $f$.
    The same is true for $1/g$, and it has the same type as $1/f$.

    If both $1/f,1/g$ are of the first type, then they are forced to be linear polynomials, because $1=cf'$ can be true only in that case.

    If both $1/f, 1/g$ are of the second type, then they are of degree $d= 2$. Indeed, they must be of degree $d\ge 2$, because they split, and if either is of degree $d>2$, then by the Ramen lemma \ref{lem: ramen lemma}, we get a contradiction with Corollary \ref{cor: degree no more than 2}. Therefore, both polynomials have degree $2$. Thus, $f_1=(x-\alpha)$, $f_2=(x-\beta)$, and $c(f_1'f_2-f_1f_2')=c((x-\beta)-(x-\alpha))=c(\alpha-\beta)=1$.

    We observe, that $c(\alpha-\beta)=1=\overline{1}=\overline{c}(\beta-\alpha)$, so $c+\overline{c}=0$. Therefore, any quadratic polynomial works, but we need the $c$ for $f$ and the $c$ for $g$ to give the same quadratic field. So, if the splitting field of $f$ is that of $g$, then we are done.
\end{example}

\begin{remark}
    In the above example \ref{Ex: Irreducible Monic Polynomial Coefficients}, the assumption about being monic is irrelevant, because we have a rational toric invariance on our coefficients, Lemma \ref{lem: toric invariance}. Thus, WLOG, we can always make the coefficients monic.
\end{remark}

\subsubsection{The final first example} We started with an example. We will finish with an example. 

\begin{remark}[My lucky example]
    The following example \ref{ex: x^n+a} is the one I computed by hand first while exploring how hard it is to compute $p$-powers of derivations on products of curves; through this exercise, I convinced myself that this paper is possible. A posteriori, this example is simple, but it touches much of the theory from the previous sections; yet it does not require it to be fully developed, and it allows for some tricks that are not obvious or general, so it was a great lucky choice.
\end{remark}

\begin{example}\label{ex: x^n+a}
Let $n\ge 0$ be an integer, and let $a,b\in \mathbb{Q}$ be two rational numbers.

We consider the following derivations:
\[
D= (x^n+a)\frac{\partial}{\partial x} +(y^n+b)\frac{\partial}{\partial y}.
\]

We provide first integrals for it case by case.
\begin{enumerate}
    \item $n=0$: $F=(1+b)x-(1+a)x$ works. It is not constant if and only if $D\ne 0$.
    \item $n=1$: $F=\frac{x+a}{y+b}$ works.
    \item $n\ge 3$: This forces $a=b=0$. In this case we can take $F=\frac{1}{x^{n-1}}-\frac{1}{y^{n-1}}$.

    If $a\ne 0$ or $b\ne 0$, then we get a contradiction:
    
    If $a\ne 0$, then $x^n+a$ is not constant-type, Definition \ref{def: constant-type}. And this is a necessary condition.
    
    Indeed, since $n\ge 3$, we claim that there are infinitely many primes $p$ with $p>n$ such that $n$ divides one of $2p-1, 3p-1,\ldots, (n-1)p-1$, which follows from the Dirichlet theorem on primes in arithmetic progressions: since $n\ge 3$, Euler's totient function satisfies $\phi(n)\ge 2$. Let $1<d<n$ be an integer such that $(d,n)=1$.
    Then there exist infinitely many primes $p$ such that $nk+d=p$; thus there exist infinitely many primes $p(1/d)\equiv 1$ modulo $n$, and $1/d$ is congruent to one of $2,3,\ldots,n-1$ modulo $ n$. Thus, our claim.

    This means by Proposition \ref{prop: computing p-powers on curves}, that for infinitely many $p$ we have
    \[
    D^p=A(x^n+a)\frac{\partial}{\partial x} +B(y^n+b)\frac{\partial}{\partial y}
    \]
    with $A$ non-constant function in $x$. This means it is not $p$-closed for this $p$, so it is not algebraically integrable by Proposition \ref{prop: easy direction AlgIntConj}. Which, because the rank is $1$, means there is no non-constant rational function killed by $D$.

    \item $n=2$: By Theorem \ref{thm: main theorem n>=2} we need $a=b=0$, or both $a,b$ not zero, and $a/b$ a square.
    
    In the first case we can take $F=\frac{1}{x}-\frac{1}{y}$.
    In the second case, we need to know the residues. It is easy to compute them: $\frac{\pm 1}{2\sqrt{-a}}$ and $\frac{\pm 1}{2\sqrt{-b}}$. The condition on the dimension of $V$ being $1$ is that $a/b$ is a square of a rational number; then we can form an exponent of integrals that is a rational function killed by $D$.

    Here, we outline how one could get to the existence of $F$ a bit differently. It is natural to transform the question $D(F)=0$ into $dF=\omega$, where $\omega=\frac{dx}{x^2+a}-\frac{dy}{y^2+b}$, because this is a classical ODE trick. Thus, one can find an analytical $F$ by solving it directly. One gets:
    \[
    F=\frac{1}{\sqrt{a}}\op{arctan}(\frac{x}{\sqrt{a}})-\frac{1}{\sqrt{b}}\op{arctan}(\frac{y}{\sqrt{b}}).
    \]
    Now, if $a/b$ is a square, then $\op{tan}(\sqrt{a}F)$ is a non-constant algebraic function in $x,y$, because of the classical formulas for $\op{tan}(nx)$ in terms of $\op{tan}(x)$, and for $\op{tan}(x+y)$ in terms of $\op{tan}(x)$ and $\op{tan}(y)$. However, it might not be a rational function; it might be wild. Nevertheless, Lemma \ref{lem: eventually alg is alg: rank 1 for dim 2} implies that a non-constant rational function exists. 
\end{enumerate}
\end{example}

\bibliographystyle{alphaurl}
\bibliography{bib}

\end{document}